\documentclass[12pt]{amsart}

\usepackage{amssymb}

\usepackage{enumerate}

\makeatletter
\@namedef{subjclassname@2010}{%
  \textup{2010} Mathematics Subject Classification}
\makeatother

\newtheorem{thm}{Theorem}[section]
\newtheorem{cor}[thm]{Corollary}
\newtheorem{lem}[thm]{Lemma}

\theoremstyle{definition}
\newtheorem{defin}[thm]{Definition}
\newtheorem{rem}[thm]{Remark}

\numberwithin{equation}{section}

\newcommand\rn{\mathbb R^n}
\newcommand\re{\mathbb R}
\newcommand\n{\mathbb N}
\newcommand\z{\mathbb Z}
\newcommand\D{\mathbb D}

\newcommand\ph{\varphi}
\newcommand\eps{\varepsilon}

\newcommand\diam{\operatorname{diam}}

\newcommand\cH{\mathcal H}
\newcommand\rv{\overline{\mathbb R}}

\newcommand\cA{\mathcal A}

\providecommand{\ch}[1]{\text{\raise 2pt \hbox{$\chi$}\kern-0.2pt}_{#1}}

\providecommand{\vint}[1]{\mathchoice
          {\mathop{\vrule width 5pt height 3 pt depth -2.5pt
                  \kern -9pt \kern 1pt\intop}\nolimits_{\kern -5pt{#1}}}%
          {\mathop{\vrule width 5pt height 3 pt depth -2.6pt
                  \kern -6pt \intop}\nolimits_{\kern -3pt{#1}}}%
          {\mathop{\vrule width 5pt height 3 pt depth -2.6pt
                  \kern -6pt \intop}\nolimits_{\kern -3pt{#1}}}%
          {\mathop{\vrule width 5pt height 3 pt depth -2.6pt
                  \kern -6pt \intop}\nolimits_{\kern -3pt{#1}}}}

\begin{document}


\baselineskip=17pt



\title[Besov Capacity in Metric Spaces]{The Besov Capacity in Metric Spaces}

\author[J. Nuutinen]{Juho Nuutinen}
\address{Department of Mathematics and Statistics P.O. Box 35\\ University of Jyv\"askyl\"a\\
FI-40014, Finland}
\email{juho.nuutinen@jyu.fi}


\date{}

\begin{abstract}
We study a capacity theory based on a definition of Haj\l asz--Besov functions. We prove several properties of this capacity in the general setting of a metric space equipped with a doubling measure. 
The main results of the paper are lower bound and upper bound estimates for the capacity in terms of a modified Netrusov--Hausdorff content. Important tools are $\gamma$-medians, for which we also prove a new version of a Poincar\'e type inequality.
\end{abstract}

\subjclass[2010]{Primary 31E05; Secondary 31B15}

\keywords{Besov spaces, capacity, metric spaces}

\maketitle

\section{Introduction}
In this paper, we study a metric version of the Besov capacity in a metric measure space $(X,d, \mu)$ with a doubling measure $\mu$. Different capacities in the metric setting have been studied previously, for example, in \cite{BB}, \cite{B}, \cite{GT}, \cite{HaKi}, \cite{KiMa}, \cite{L} and \cite{NS}. In the Euclidean setting, the Besov capacity has been studied, for example, in \cite{A1}, \cite{A2}, \cite{AH}, \cite{AHS}, \cite{AX}, \cite{D}, \cite{HN}, \cite{MX}, \cite{Ne90}, \cite{Ne92}, \cite{Ne96} and \cite{Sto}. Our definition of the Besov capacity is based on the pointwise definition of fractional $s$-gradients and the Haj\l asz--Besov space $N^s_{p,q}(X)$, $0<s<\infty$ and $0<p,q\leq \infty$. This characterization for Besov spaces was first introduced in \cite{KYZ} and has recently been applied, for example, in \cite{GKZ}, \cite{HIT}, \cite{HKT}, \cite{HeTu} and \cite{HeTu2}. The Haj\l asz--Besov space $N^s_{p,q}(X)$ consists of $L^p$-functions $u$ that have a fractional $s$-gradient with finite mixed $l^q(L^p(X))$-norm. A sequence of nonnegative measurable functions $(g_k)_{k\in\z}$ is a fractional $s$-gradient of $u$, if it satisfies the Haj\l asz type pointwise inequality
$$|u(x)-u(y)|\le d(x,y)^s(g_k(x)+g_k(y))$$
for all $k\in\z$ and almost all $x,y\in X$ satisfying $2^{-k-1}\le d(x,y)<2^{-k}$. We give the precise definitions and notation in Section 2, where we also prove two useful lemmas for the fractional $s$-gradients. 

In Section 3, we define the Besov capacity and prove the basic properties of this capacity. These properties include monotonicity, a version of subadditivity and several convergence results. In particular, we apply the results to show that Haj\l asz--Besov functions are quasicontinuous with respect to this capacity.
Besov capacity has been studied previously in the metric setting in \cite{B} and \cite{Co}. However, in these papers only the case $p=q$ is considered in a less general metric space. Ahlfors $Q$-regularity, for example, is assumed in both papers. In a recent preprint \cite{HKT}, some of the results of this section are stated and a version of subadditivity is proved. We prove several new results and give full proofs to the basic properties of the Besov capacity not proved in \cite{HKT}.

The $\gamma$-medians are extremely useful tools in the setting of Besov spaces, especially when $0<p\leq 1$ or $0<q\leq 1$. In our proofs, they take the place of integral averages. Medians behave similarly as the integral averages, but have the advantage that the function does not need to be locally integrable. In Section 4, we study some of the basic properties of the $\gamma$-medians that we later use in our proofs. One of the main results of this section is a new Sobolev--Poincar\'e type inequality for the $\gamma$-medians. For slightly different results, see \cite{HKT} and \cite{HeTu2}. Also, we recall the definition of discrete median convolutions and use them as tools to obtain Theorem \ref{dc2}, which says that for compact sets it is equivalent to consider only the locally Lipschitz admissible functions when calculating the capacity.

In Section 5, we study a modified version of the Netrusov--Hausdorff content. The Netrusov--Hausdorff content was introduced in $\mathbb{R}^n$ by Netrusov in \cite{Ne92} and \cite{Ne96}. It has also been studied, for example, in \cite{A2} and \cite{HN}. We modify the Euclidean definition to the metric setting, since in our case the dimension of the space $X$ does not need to be constant. Instead of summing over the powers of the radii $r_j$ of the balls in the covering, we sum over the measures of the balls in the covering divided by the values $\phi(r_j)$ of an increasing function $\phi$. Our main results are lower bound and upper bound estimates for the capacity in terms of the modified Netrusov--Hausdorff content (see Theorem \ref{Cap < NH} and Theorem \ref{toinen suunta}). 

\section{Notation and preliminaries}
\subsection{Basic assumptions and notation}
We assume that the triple  $(X, d,\mu)$, denoted simply by $X$, is a metric measure space equipped with a metric $d$ and a Borel regular,
doubling outer measure $\mu$, for which the measure of every ball is positive and finite.
The doubling property means that there is a fixed constant $c_d>0$, called {\em the doubling constant}, such that
$$\mu(B(x,2r))\le c_d\mu(B(x,r))$$
for every ball $B(x,r)=\{y\in X:d(y,x)<r\}$, where $x\in X$ and $r>0$. 

We denote {\em the integral average} of a locally integrable function $u$ over a set $A$ of positive and finite measure by
\[
u_A=\vint{A}u\,d\mu=\frac{1}{\mu(A)}\int_A u\,d\mu.
\]

By $\ch{E}$ we denote the characteristic function of a set $E\subset X$ and by $\rv$ the extended real numbers $[-\infty,\infty]$.
We denote the set of all measurable, almost everywhere finite functions $u\colon X\to\rv$ by $L^0(X)$.
In general, $C$ is a positive constant whose value is not necessarily the same at each occurrence.

\subsection{Fractional $s$-gradients and Haj\l asz--Besov spaces} 
We define the Haj\l asz--Besov space in terms of pointwise inequalities, as in \cite{KYZ}. This characterization is motivated by the definition of a generalized gradient and Haj\l asz--Sobolev space $M^{s,p}(X)$, defined for $s=1$, $p\ge1$ in \cite{H} and for fractional scales in \cite{Y}. There are also other definitions for Besov spaces in the metric setting. They have been studied, for example, in \cite{GKS}, \cite{GKZ}, \cite{HMY}, \cite{KYZ}, \cite{MY}, \cite{SYY}, \cite{YZ}.

\begin{defin}\label{fract grad}
Let $0<s<\infty$.
A sequence of nonnegative measurable functions $(g_k)_{k\in\z}$ is a  {\em fractional $s$-gradient} of a function
$u\in L^0(X)$, if there exists a set $E$ with $\mu(E)=0$ such that
\begin{equation}\label{frac grad}
|u(x)-u(y)|\le d(x,y)^s(g_k(x)+g_k(y))
\end{equation}
for all $k\in\z$ and all $x,y\in X\setminus E$ satisfying $2^{-k-1}\le d(x,y)<2^{-k}$.
The collection of all fractional $s$-gradients of $u$ is denoted by $\D^s(u)$.
\end{defin}

We prove two lemmas that we use later in the proofs of this paper. The above definition implies the following lattice property for fractional $s$-gradients.

\begin{lem}\label{max lemma}
Let $0<s<\infty$, $u,v\in L^0(X)$, $(g_k)_{k\in \mathbb{Z}}\in\D^s(u)$ and $(h_k)_{k\in \mathbb{Z}}\in\D^s(v)$. Then,
the sequence $(\max\{g_k,h_k\})_{k\in \z}$ is a fractional $s$-gradient of $\max\{u,v\}$ and $\min\{u,v\}$.
\end{lem}

\begin{proof}
We define $w=\max\{u,v\}$ and assume that $G$ and $H$ are the exeptional sets for $(g_k)_{k\in \mathbb{Z}}$ and $(h_k)_{k\in \mathbb{Z}}$ in the Definition \ref{fract grad}. Clearly, the function $w$ is measurable and $(\max\{g_k,h_k\})_{k\in \z}$ is a sequence of nonnegative measurable functions. We show that \eqref{frac grad} holds outside the set $G \cup H$ of measure zero. Let $$F_u=\{x \in X\setminus(G \, \cup \, H): u(x)\geq v(x)\}$$ and $$F_v=\{x \in X\setminus(G\cup H): u(x)<v(x)\}.$$ If $x$, $y \in F_u$ then $$|w(x)-w(y)|=|u(x)-u(y)| \leq d(x,y)^s(g_k(x)+g_k(y))$$
for all $k\in\z$ satisfying $2^{-k-1}\le d(x,y)<2^{-k}$. Similarly, for $x$, $y \in F_v$ we get $$|w(x)-w(y)| \leq d(x,y)^s(h_k(x)+h_k(y))$$ for all $k\in\z$ satisfying $2^{-k-1}\le d(x,y)<2^{-k}$.

If $x \in F_u$ and $y\in F_v$, then we can look at the two cases $u(x)\geq v(y)$ and $u(x)<v(y)$ separately. In the first case
\begin{align*}
|w(x)-w(y)|&=|u(x)-v(y)|=u(x)-v(y) \\
&\leq u(x)-u(y) \leq d(x,y)^s(g_k(x)+g_k(y))
\end{align*}
for all $k\in\z$ satisfying $2^{-k-1}\le d(x,y)<2^{-k}$. In the second case 
\begin{align*}
|w(x)-w(y)|&=v(y)-u(x)\leq v(y)-v(x) \\
&\leq d(x,y)^s(h_k(x)+h_k(y))
\end{align*}
for all $k\in\z$ satisfying $2^{-k-1}\le d(x,y)<2^{-k}$. The case $x \in F_v$ and $y\in F_u$ follows by symmetry, and hence $$|w(x)-w(y)| \leq d(x,y)^s(\max\{g_k,h_k\}(x)+\max\{g_k,h_k\}(y))$$ for all $k\in\z$ and all $x,y\in X\setminus (G \, \cup \, H)$ satisfying $2^{-k-1}\le d(x,y)<2^{-k}$. The proof for the function $\min\{u,v\}$ follows along the same lines.
\end{proof}

The next lemma is useful when we want to show that the supremum of countably many Haj\l asz--Besov functions belongs to the Haj\l asz--Besov space $N_{p,q}^s(X)$ (see Definition \ref{Besov}). 

\begin{lem}\label{sup lemma} 
Let $u_i\in L^0(X)$ and $(g_{i,k})_{k\in\z}\in \D^s(u_i)$, $i\in\n$. 
Let $u=\sup_{i\in \n} u_i$ and $(g_k)_{k\in\z}=(\sup_{i\in\n}g_{i,k})_{k\in\z}$.
If $u\in L^0(X)$, then $(g_k)_{k\in\z}\in \D^s(u)$.
\end{lem}

\begin{proof}
Since $u \in L^0(X)$, it is finite almost everywhere. Let $x$, $y \in X\setminus E$, with $u(y)\leq u(x)<\infty$, where $E$ is the union of exceptional sets for the functions $u_i$ in the Definition \ref{fract grad}. Let $\epsilon>0$. There is $i=i_x \in \mathbb{N}$, such that $u(x)<u_i(x)+\epsilon$. Now, since $u(y)\geq u_i(y)$, we have
\begin{align*}
|u(x)-u(y)| &= u(x)-u(y) \leq u_i(x)+\epsilon-u_i(y) \\
&\leq d(x,y)^s \left(g_{i,k}(x)+g_{i,k}(y)\right) + \epsilon \\
&\leq d(x,y)^s \left(g_k(x)+g_k(y)\right) + \epsilon
\end{align*}
for all $k\in\z$ satisfying $2^{-k-1}\le d(x,y)<2^{-k}$. Letting $\epsilon \rightarrow 0$ proves the claim.
\end{proof}

For $0<p,q\le\infty$ and a sequence $(f_k)_{k\in\z}$ of measurable
functions, we define
\[
\big\|(f_k)_{k\in\z}\big\|_{l^q(L^p(X))}
=\big\|\big(\|f_k\|_{L^p(X)}\big)_{k\in\z}\big\|_{l^q},
\]
where
\[
\big\|(a_k)_{k\in\z}\big\|_{l^{q}}
=
\begin{cases}
\big(\sum_{k\in\z}|a_{k}|^{q}\big)^{1/q},&\quad\text{when }0<q<\infty, \\
\;\sup_{k\in\z}|a_{k}|,&\quad\text{when }q=\infty.
\end{cases}
\]

\begin{defin}\label{Besov}
Let $0<s<\infty$ and $0<p,q\le\infty$.
The  {\em homogeneous Haj\l asz--Besov space} $\dot N_{p,q}^s(X)$ consists of functions
$u\in L^0(X)$, for which the (semi)norm
\[
\|u\|_{\dot N_{p,q}^s(X)}=\inf_{(g_k)\in\D^s(u)}\|(g_k)\|_{l^q(L^p(X))}
\]
is finite. The {\em Haj\l asz--Besov space} $N_{p,q}^s(X)$ is $\dot N_{p,q}^s(X)\cap L^p(X)$
equipped with the norm
\[
\|u\|_{N_{p,q}^s(X)}=\|u\|_{L^p(X)}+\|u\|_{\dot N_{p,q}^s(X)}.
\]
\end{defin}
For $0<s<1$ and $0<p,q\le \infty$, the space $N^s_{p,q}(\rn)$ coincides with the classical Besov space defined via differences ($L^p$-modulus of smoothness), see \cite{GKZ}. When $0<p<1$ or $0<q<1$, the (semi)norms defined above are actually quasi-(semi)norms, but for simplicity we call them 
just norms. Recall that a quasinorm is similar to a norm in that it satisfies the norm axioms, except that there is a constant $C >1$ on the right-hand side of the triangle inequality. 

\subsection{Inequalities} 
We will often use the following elementary inequality
\begin{equation}\label{elem ie}
\sum_{i\in\z} a_i\le \Big(\sum_{i\in\z} a_i^{\beta}\Big)^{1/\beta},
\end{equation}
which holds whenever $a_i\ge 0$ for all $i$ and $0<\beta\le 1$.
H\"older's inequality for sums (when $1<b<\infty$) and \eqref{elem ie} imply the next lemma that we use later to estimate the norms of fractional gradients.

\begin{lem}\label{summing lemma}\cite[Lemma 3.1]{HIT}
Let $1<a<\infty$, $0<b<\infty$ and $c_k\ge 0$, $k\in\z$. There exists a constant $C=C(a,b)$ such that
\[
\sum_{k\in\z}\Big(\sum_{j\in\z}a^{-|j-k|}c_j\Big)^b\le C\sum_{j\in\z}c_j^b.
\]
\end{lem}

\section{Capacity}

In this section, we study a metric version of the Besov capacity. We prove the basic properties of this capacity, including several useful lemmas and convergence results. In particular, we show that Haj\l asz--Besov functions $u \in N^s_{p,q}(X)$, $0<s<1$ and $0<p, q <\infty$, are quasicontinuous with respect to the capacity (see Theorem \ref{kvasijvuus}). 
Recently, some of the results of this section have been stated or proved in \cite{HKT}. We give complete proofs to the results not proved there as well as to completely new ones. 

\begin{defin}
Let $0<s<\infty$ and $0<p,q\le \infty$. 
The {\em Besov capacity} of a set $E\subset X$ is
\[
C_{p,q}^s(E)=\inf\Big\{\|u\|_{N_{p,q}^s(X)}^p: u\in\mathcal A(E)\Big\},
\]
where 
\[
\mathcal A(E)=\{u\in N_{p,q}^s(X): u\ge 1 \text{ in a neighbourhood of } E\}
\] 
is the set of {\em admissible functions} for the capacity.
We say that a property holds \emph{$C_{p,q}^s$-quasieverywhere} if it holds outside a set of $C_{p,q}^s$-capacity zero.
\end{defin}

\begin{rem} Lemma \ref{max lemma} implies that
\[
C_{p,q}^s(E)=\inf\Big\{\|u\|_{N_{p,q}^s(X)}^p: u\in\mathcal A'(E)\Big\},
\]
where
$\mathcal A'(E)=\{u\in \cA(E): 0\le u\le 1\}$. Since $\mathcal A'(E) \subset \mathcal A(E)$, we have that $C_{p,q}^s(E) \leq \inf\left\{\|u\|_{N_{p,q}^s(X)}^p: u\in\mathcal A'(E)\right\}$. To prove the inequality to the reverse direction, let $\epsilon > 0$ and let $u \in  \mathcal A(E)$ be such that $$\|u\|_{N_{p,q}^s(X)}^p \leq C_{p,q}^s(E) + \epsilon.$$ Then $v= \max\{0, \min\{u,1\}\} \in \mathcal A'(E)$ and by Lemma \ref{max lemma} we have $\D^s(u) \subset \D^s(v)$. Now $$\inf\left\{\|w\|_{N_{p,q}^s(X)}^p: w\in\mathcal A'(E)\right\} \leq \|v\|_{N_{p,q}^s(X)}^p \leq \|u\|_{N_{p,q}^s(X)}^p \leq C_{p,q}^s(E) + \epsilon$$ and letting $\epsilon \rightarrow 0$ yields the inequality.
\end{rem}

\begin{rem}
It follows immediately that $$\mu(E) \leq C_{p,q}^s(E),$$ for every $E \subset X$. Let $u \in \mathcal{A}(E)$. Then, there is an open set $U \supset E$, such that $u \geq 1$ in $U$. Hence $$\mu(E) \leq \mu(U) \leq \|u\|_{L^p(X)}^p \leq \|u\|_{N_{p,q}^s(X)}^p$$ and taking infimum over all $u \in \mathcal{A}(E)$ proves the inequality.
\end{rem}

The $C_{p,q}^s$-capacity is generally not an outer measure. The definition clearly implies monotonicity but the capacity is not necessarily subadditive.
However, for practical purposes it is enough that the capacity
satisfies \eqref{eq: r-subadd} for some $r>0$. Even in the Euclidean setting, countable subadditivity for the Besov capacity is known only when $p\le q$ (see \cite{A1}).


\begin{thm}\label{r-subadd}\cite[Lemma 6.4]{HKT}
Let $0<s<\infty$ and $0<p,q\le \infty$. 
Then there are constants $C \ge 1$ and $0<r\le 1$ such that
\begin{equation}\label{eq: r-subadd}
C_{p,q}^s\big(\bigcup_{i\in\n}E_i\big)^r\le C \sum_{i\in\n} C_{p,q}^s(E_i)^r
\end{equation}
for all sets $E_i\subset X$, $i\in\n$.
Actually, \eqref{eq: r-subadd} holds with $r=\min\{1,q/p\}$.
\end{thm}

The Besov capacity is an outer capacity. This means that the capacity of a set $E \subset X$ can be obtained by approximating $E$ with open sets from the outside.

\begin{lem}\label{cap remark 2} The $C_{p,q}^s$-capacity is an outer capacity, that is,
\[
C_{p,q}^s(E)=\inf\big\{C_{p,q}^s(U): U\supset E,\ U \text{ open}\big\}.
\]
\end{lem}

\begin{proof}
By the monotonicity, $C_{p,q}^s(E)\leq\inf\big\{C_{p,q}^s(U): U\supset E,\ U \text{ open}\big\}$. To obtain the reverse inequality, let $\epsilon >0$ and let $u \in \mathcal A(E)$ be such that $$\|u\|_{N_{p,q}^s(X)}^p \leq C_{p,q}^s(E) + \epsilon.$$ Now, since $u$ is an admissible function for the capacity, there is an open set $U$ containing $E$ such that $u \geq 1$ on $U$. Then $$C_{p,q}^s(U) \leq \|u\|_{N_{p,q}^s(X)}^p \leq C_{p,q}^s(E) + \epsilon.$$ Letting $\epsilon \rightarrow 0$ proves the claim.
\end{proof}

The following compatibility condition says that removing a set of measure zero does not change the capacity of an open set. In particular, this result can be applied to prove a uniqueness result for $C_{p,q}^s$-quasicontinuous representatives of a Haj\l asz--Besov function (see Remark \ref{yksikas}).

\begin{lem}\label{rigidity}
Let $0<s<\infty$ and $0<p,q\le \infty$. If $U$ is an open set and $\mu(E)=0$, then $$C_{p,q}^s(U)=C_{p,q}^s(U\setminus E).$$
\end{lem}
\begin{proof}
Clearly, by monotonicity, $C_{p,q}^s(U) \geq C_{p,q}^s(U \setminus E)$ so it remains to show the inequality to the other direction. Let $\epsilon >0$ and 
let $u \in \cA'(U\setminus E)$, with $(g_{k})_{k \in \mathbb{Z}} \in \D^s(u)$, be such that $\chi_{U\setminus E} \leq u \leq 1$ and
\[
\big(\|u\|_{L^p(X)}+\|(g_{k})\|_{l^q(L^p(X))}\big)^{p}< C_{p,q}^s(U\setminus E)+\eps.
\] 
Let $v$ be a function, such that $v=u$ in $X \setminus U$ and $v=1$ in $U$. Then $v=u$ outside the set $U \cap E$, which has measure zero, and so $\|v\|_{L^p(X)}=\|u\|_{L^p(X)}$. Also, $(g_k)_{k \in \mathbb{Z}} \in \D^s(v)$, since we can choose the exceptional set in Definition \ref{fract grad} to be the union of $U \cap E$ and the exceptional set related to $u$ and $(g_k)_{k \in \mathbb{Z}}$. Then, $v \in \cA'(U)$ and
\begin{align*} 
C_{p,q}^s(U) \leq \|v\|_{N_{p,q}^s(X)}^p &\leq (\|u\|_{L^p(X)}+\|(g_{k})\|_{l^q(L^p(X))}\big)^{p} \\ &<C_{p,q}^s(U\setminus E) + \epsilon
\end{align*}
and letting $\epsilon \rightarrow 0$ proves the claim.
\end{proof}

The outer capacity property of the Besov capacity implies the next convergence result for compact sets.

\begin{thm}
If $X \supset K_1 \supset  K_2 \supset \cdots$ are compact sets and $K=\bigcap_{i=1}^{\infty}K_i$, then $$\lim_{i \rightarrow \infty} C_{p,q}^s(K_i)= C_{p,q}^s(K).$$
\end{thm}

\begin{proof}
Clearly, by the monotonicity, $\lim_{i \rightarrow \infty} C_{p,q}^s(K_i) \geq C_{p,q}^s(K)$ and so it remains to show the inequality to the other direction. If $U$ is an open set containing $K$, then $U \cup \bigcup_{i=1}^{\infty} (X \setminus K_i)$ is an open cover of the set $K_1$ and, since $K_1$ is compact, there is a finite subcover, i.e. a positive integer $N$, such that $$K_1 \subset U \cup \bigcup_{i=1}^N (X \setminus K_i)= U \cup (X \setminus K_N).$$ It follows that $K_N \subset U$, since $K_N \subset K_1$.
Hence, 
$\lim_{i \rightarrow \infty} C_{p,q}^s(K_i) \leq C_{p,q}^s(U)$ and by Lemma \ref{cap remark 2} we obtain $$\lim_{i \rightarrow \infty} C_{p,q}^s(K_i) \leq \inf\{C_{p,q}^s(U): U \supset K,\ U \text{ open} \} = C_{p,q}^s(K).$$
\end{proof}


We apply the following theorem to show that Haj\l asz--Besov functions are quasicontinuous with respect to the $C^s_{p,q}$-capacity (see Theorem \ref{kvasijvuus}).

\begin{thm}
Let $0<s<\infty$ and $0<p,q\le \infty$. If $(u_i)_{i \in \mathbb{N}}$ is a Cauchy sequence of continuous functions in $N^s_{p,q}(X)$, then there is a subsequence of $(u_i)_{i \in \mathbb{N}}$ which converges pointwise $C_{p,q}^s$-quasieverywhere in $X$. Moreover, the convergence is uniform outside a set of arbitrary small $C_{p,q}^s$-capacity.
\end{thm}

\begin{proof}
Let $r=\min\{1, q/p\}$. There is a subsequence of $(u_i)_{i \in \mathbb{N}}$, which we still denote by $(u_i)_{i \in \mathbb{N}}$, such that 
\begin{equation}\label{haha}
\sum_{i=1}^{\infty} 2^{ipr} \, \|u_i - u_{i+1} \|^{pr}_{N_{p,q}^s(X)}<\infty.
\end{equation}
For $i \in \mathbb{N}$, let $$A_i= \{x \in X: |u_i(x)- u_{i+1}(x)|>2^{-i}\}$$ and $B_j= \bigcup_{i=j}^{\infty} A_i$. Since the functions $u_i$ are continuous, the sets $A_i$ and $B_j$ are open. It follows that the function $2^i|u_i-u_{i+1}|$ is admissible for the Besov capacity of $A_i$ and $$C_{p,q}^s(A_i) \leq 2^{ip} \, \|u_i - u_{i+1} \|^p_{N_{p,q}^s(X)}.$$ Now, by Theorem \ref{r-subadd}, we get $$C_{p,q}^s(B_j) \leq C \Big(\sum_{i=j}^{\infty} C_{p,q}^s(A_i)^r\Big)^{1/r}\leq C \Big(\sum_{i=j}^{\infty} 2^{ipr} \, \|u_i - u_{i+1} \|^{pr}_{N_{p,q}^s(X)}\Big)^{1/r}.$$ Since $B_1 \supset B_2 \supset \cdots$ and sum \eqref{haha} converges, we have that $$C_{p,q}^s\Big(\bigcap_{j=1}^{\infty} B_j \Big) \leq \lim_{j\rightarrow \infty} C_{p,q}^s(B_j)=0$$ and $(u_i)_{i \in \mathbb{N}}$ converges pointwise in $X\setminus\bigcap_{j=1}^{\infty} B_j$. Moreover, $$|u_j(x) -u_k(x)| \leq \sum_{i=j}^{k-1} |u_i(x)-u_{i+1}(x)| \leq \sum_{i=j}^{k-1}2^{-i} \leq 2^{1-j}$$ for all $x \in X\setminus B_j$, for every $k>j$. Hence, the convergence is uniform in $X\setminus B_j$ and the claim follows.
\end{proof}

\begin{defin}
A function $u \colon X\to\rv$ is $C_{p,q}^s$-\emph{quasicontinuous}, if for every $\eps>0$ there exists a set $U$ such that $C_{p,q}^s(U)<\eps$ and the restriction of $u$ to $X\setminus U$ is continuous. 
\end{defin}

\noindent Note that, by Lemma \ref{cap remark 2}, the set $U$ can be chosen to be open.

\begin{thm}\label{kvasijvuus} 
Let $0<s<1$ and $0<p,q<\infty$. Then, for every $u\in N_{p,q}^s(X)$, there exists a $C_{p,q}^s$-quasicontinuous function $v$ such that $u=v$ almost everywhere.
\end{thm}

\begin{proof}
Since continuous functions are dense in $N_{p,q}^s(X)$, when $0<s<1$ and $0<p,q<\infty$, (see \cite[Theorem 1.1]{HKT}) 
and $N_{p,q}^s(X)$ is complete by the appendix in \cite{HeTu2}, the claim follows from the previous theorem. Indeed, $u \in N_{p,q}^s(X)$ if and only if there is a sequence $(u_i)_{i \in \mathbb{N}}$ of continuous functions in $L^p(X)$ and $(g_{i,k})_{k\in\z}\in \D^s(u_i-u)$, such that $u_i \rightarrow u$ in $L^p(X)$ and $\|(g_{i,k})_{k\in\z}\|_{l^q(L^p(X))} \rightarrow 0$. By the previous theorem, the limit function is $C_{p,q}^s$-quasicontinuous.
\end{proof}

\begin{rem}\label{yksikas}
The $C_{p,q}^s$-quasicontinuous representative is unique in the sense that if two $C_{p,q}^s$-quasi\-con\-ti\-nuous functions coincide almost everywhere, then they actually coincide outside a set of $C_{p,q}^s$-capacity zero. This follows from Lemmas \ref{cap remark 2} and \ref{rigidity}, and from a nice argument, in an abstract setting, in \cite{K}.
\end{rem}

\section{$\gamma$-median}

In this section, we study $\gamma$-medians that are important tools in our setting of Besov spaces. In our proofs, they take the place of integral averages and are extremely useful when $0<p\leq1$ or $0<q\leq1$. One of the main results is Theorem \ref{mm} which is a new Sobolev--Poincar\'e type inequality for the medians. Recently, slightly different results have been proved in \cite{HKT} and \cite{HeTu2}, where an additional nonempty spheres property is also assumed on the underlying space $X$. In the latter part of this section, we define a discrete median convolution which we use to show that it is equivalent to consider only the locally Lipschitz admissible functions when calculating the capacity of a compact set (see Theorem \ref{dc2}). These results are useful in Section 5, where we study a modified Netrusov--Hausdorff content related to the capacity.

Next, we define the $\gamma$-median of a function $u \in L^0(X)$ over a set of finite measure. Previously, the $\gamma$-medians have been studied, for example, in \cite{Fu}, \cite{GKZ}, \cite{JPW}, \cite{JT}, \cite{Le}, \cite{LP}, \cite{PT}, \cite{St}, \cite{Zh}, and recently in \cite{HIT}, \cite{HKT} and \cite{HeTu2}.

\begin{defin}\label{median}
Let $0<\gamma\le 1/2$. The \emph{$\gamma$-median} of a function
$u\in L^0(X)$ over a set $A$ of finite measure is
\[
m_u^\gamma(A)=\inf\big\{a\in\mathbb{R}: \mu(\{x\in A: u(x)>a\})< \gamma\mu(A)\big\}.
\]
\end{defin}

In the following lemma, we give some basic properties of the $\gamma$-median. 

\begin{lem}\label{median lemma} 
Let $A\subset X$ be a set with $\mu(A)<\infty$. 
Let $u,v\in L^0(A)$ and let $0<\gamma\le1/2$. 
The $\gamma$-median has the following properties:
\begin{itemize}
\item[a)] If $\gamma\le\gamma'$, then $m_{u}^{\gamma}(A)\ge m_{u}^{\gamma'}(A)$.
\item[b)] If $u\le v$ almost everywhere, then  $m_{u}^{\gamma}(A)\le m_{v}^{\gamma}(A)$.
\item[c)] If $A\subset B$ and $\mu(B)\le C\mu(A)$, then $m_{u}^{\gamma}(A)\le m_{u}^{\gamma/C}(B)$.
\item[d)] If $c\in\mathbb{R}$, then $m_u^\gamma(A)+c=m_{u+c}^\gamma(A)$.
\item[e)] If $c\in\mathbb{R}$, then $m_{c\,u}^\gamma(A)=c\,m_{u}^\gamma(A)$.
\item[f)] $|m_{u}^\gamma(A)|\le m_{|u|}^\gamma(A)$. 
\item[g)] For every $p>0$ and $u\in L^p(A)$, 
\[
m_{|u|}^\gamma(A)\le \Big(\gamma^{-1}\vint{A}|u|^p\,d\mu\Big)^{1/p}.
\]
\item[h)] If $u$ is continuous, then \[\lim_{r\to 0} m_{u}^\gamma(B(x,r))=u(x)\]
for every $x\in X$.
\end{itemize}
\end{lem}

\begin{proof}
We prove the property g) below. The rest of the quite straightforward proofs are left for the reader, who can also look at \cite{PT} where most of the properties are proved in the Euclidean space. The proofs in the metric setting follow essentially the same lines.

For the proof of g), we may assume that $m_{|u|}^\gamma(A)\neq 0$, since otherwise the claim is obvious. Let $p>0$ and $u \in L^p(A)$. The definition of the $\gamma$-median clearly implies that 
\begin{align*}
\gamma\mu(A) &\leq \mu(\{x\in A: |u(x)|\geq m_{|u|}^\gamma(A)\}) \\
&= \mu(\{x\in A: |u(x)|^p\geq m_{|u|}^\gamma(A)^p\})
\end{align*}
and by Chebyshev's inequality $$\mu(\{x\in A: |u(x)|^p\geq m_{|u|}^\gamma(A)^p\}) \leq \frac{1}{m_{|u|}^\gamma(A)^p} \int_{A} |u|^p \, d\mu.$$ The claim follows by combining the above two estimates.
\end{proof}

We have the following definition, that is analogous to the definition of a Lebesgue point of a function, when taking the limit of medians.

\begin{defin}\label{gLp}
Let $u\in L^0(A)$. 
A point $x$ is a \emph{generalized Lebesgue point} of $u$, if 
\[
\lim_{r\to 0} m_{u}^\gamma(B(x,r))=u(x)
\]
for all $0<\gamma\le1/2$.
\end{defin}

\begin{rem}
Recently, it was shown in \cite[Theorem 1.2]{HKT} that every point outside of a set of $C_{p,q}^s$-capacity zero of a Haj\l asz--Besov function $u$ is a generalized Lebesgue point of $u$ and that the limit of medians gives a $C_{p,q}^s$-quasicontinuous representative of the function. The result is obtained in \cite{HKT} by defining a median maximal function and using it as a tool. In particular, a capacitary weak type estimate for the median maximal function is used in the proof.
\end{rem}

The Definition \ref{fract grad} of fractional $s$-gradients implies various Sobolev--Poincar\'e type inequalities for medians. Slightly different results than the following can be found, for example, in \cite{HKT} and \cite{HeTu2}. We obtain the next theorem even without assuming a nonempty spheres property that is assumed in \cite{HKT} and \cite{HeTu2}.

\begin{thm}\label{mm} 
Let $0<\gamma\le 1/2$, $0<s,p<\infty$ and $0<q\le\infty$. Let $u\in N^s_{p,q}(X)$. Then there is a constant $C>0$ and a sequence
$(g_k)_{k \in \mathbb{Z}}\in \D^s(u)$ such that
\begin{equation}\label{joo}
\inf_{c\in\mathbb{R}}m^\gamma_{|u-c|}(B(x,2^{-k}))\le C2^{-ks}\Big(\vint{B(x,2^{-k+1})} g_k^p\,d\mu\Big)^{1/p}
\end{equation}
for every $x\in X$ and $k\in\z$. In fact, given any $(h_j)_{j \in \mathbb{Z}} \in \D^s(u)$, we can choose
\begin{equation}\label{good grad}
g_k=\Big(\sum_{j\ge k-2}2^{(k-j)s'\tilde p}h_j^p\Big)^{1/p},
\end{equation}
where $0<s'<s$ and $\tilde p=\min\{1,p\}$. Moreover, there is a constant $c>0$ such that
\begin{equation}\label{jep2}
\|(g_k)\|_{l^q(L^p(X))} \le c \, \|(h_j)\|_{l^q(L^p(X))}.
\end{equation} 
\end{thm}

\begin{proof}
Let $(h_j)_{j \in \mathbb{Z}}\in\D^s(u)$. By \cite[Lemma 2.1]{GKZ} and g) of Lemma \ref{median lemma}, there exist constants $C>0$ and $0<s'<s$, such that
\[
\inf_{c\in\mathbb{R}}m^\gamma_{|u-c|}(B(x,2^{-k}))\le C2^{-ks}\sum_{j\ge k-2}2^{(k-j)s'}\Big(\vint{B(x,2^{-k+1})} h_j^p\,d\mu\Big)^{1/p}
\]
for every $x\in X$ and $k\in\z$. We show that the right-hand side is bounded by $C2^{-ks}\left(\vint{B(x,2^{-k+1})} g_k^p\,d\mu\right)^{1/p}$,
where $$g_k=\Big(\sum_{j\ge k-2}2^{(k-j)s'\tilde p}h_j^p\Big)^{1/p}$$
and $\tilde p=\min\{1,p\}$. Notice that $(g_k)_{k \in \mathbb{Z}} \in \D^s(u)$, since 
$$|u(x)-u(y)| \leq d(x,y)^s(h_k(x)+h_k(y))\leq d(x,y)^s(g_k(x)+g_k(y))$$
for all $k \in \mathbb{Z}$ and all $x$, $y \in X \setminus E$ satisfying $2^{-k-1}\leq d(x,y)<2^{-k}$.

If $p>1$, we use H\"older's inequality for sums ($1/p+1/p'=1$)
\begin{align*}
&\sum_{j\ge k-2}2^{(k-j)s'\frac{1}{p'}} \, 2^{(k-j)s'\frac{1}{p}}\Big(\vint{B(x,2^{-k+1})} h_j^p\,d\mu\Big)^{1/p} \\ &\leq \Big(\sum_{j\ge k-2} 2^{(k-j)s'}\Big)^{1/p'} \cdot \Big(\sum_{j \ge k-2} 2^{(k-j)s'}\vint{B(x,2^{-k+1})} h_j^p\,d\mu \Big)^{1/p}  \\ &\leq C \, \Big(\vint{B(x,2^{-k+1})} \sum_{j \ge k-2} 2^{(k-j)s'} h_j^p\,d\mu \Big)^{1/p}
\end{align*}
and if $0<p\leq1$, by inequality \eqref{elem ie},
\begin{align*}
\sum_{j\ge k-2}2^{(k-j)s'}\Big(\vint{B(x,2^{-k+1})} h_j^p\,d\mu\Big)^{1/p} &\leq \Big(\sum_{j\ge k-2}2^{(k-j)s'p}\vint{B(x,2^{-k+1})} h_j^p\,d\mu\Big)^{1/p} \\
&=\Big(\vint{B(x,2^{-k+1})} \sum_{j \ge k-2} 2^{(k-j)s'p} h_j^p\,d\mu \Big)^{1/p}.
\end{align*}
Combining the two cases, we obtain inequality \eqref{joo}. To prove inequality \eqref{jep2}, we first see that
\begin{align*}
\|g_k\|_{L^p(X)}^p &= \int_X \sum_{j \ge k-2} 2^{(k-j)s'\tilde p} h_j^p\,d\mu = \sum_{j \ge k-2} 2^{(k-j)s'\tilde p} \|h_j\|_{L^p(X)}^p.
\end{align*}
Now, by Lemma \ref{summing lemma}, we get
\begin{align*}
\sum_{k \in \mathbb{Z}} \|g_k\|_{L^p(X)}^q &\leq \sum_{k \in \mathbb{Z}} \Big(\sum_{j \ge k-2} 2^{(k-j)s'\tilde p} \|h_j\|_{L^p(X)}^p\Big)^{q/p} \\
&\leq C \sum_{j \in \mathbb{Z}} \|h_j\|_{L^p(X)}^q
\end{align*}
and it follows that
\[
\|(g_k)\|_{l^q(L^p(X))} \le C \|(h_j)\|_{l^q(L^p(X))}.
\] 
\end{proof}

\begin{rem}\label{mm2}
Let $A\subset X$ be a set with $\mu(A)<\infty$, $u\in L^0(A)$ and $0<\gamma\le1/2$. Then
$$m_{|u-m^{\gamma}(A)|}^{\gamma}(A) \leq 2 \inf_{c \in \mathbb{R}} m_{|u-c|}^{\gamma}(A),$$ since for all $c \in \mathbb{R}$
\begin{align*}
m_{|u-m^{\gamma}(A)|}^{\gamma}(A) &\leq m_{|u- c| + |c-m^{\gamma}(A)|}^{\gamma}(A) = m_{|u-c|}^{\gamma}(A) + |c - m_u^{\gamma}(A)| \\
&\leq m_{|u-c|}^{\gamma}(A) + m_{|u-c|}^{\gamma}(A),
\end{align*}
where we used properties b), d) and f) of $\gamma$-median from Lemma \ref{median lemma}.
\end{rem}

Next, we define a discrete $\gamma$-median convolution that we use in the proof of Theorem \ref{dc2}. Discrete convolutions are standard tools in analysis on metric measure spaces (see, for example, \cite{CW} and \cite{MS}) and they are used, for example, to define a discrete maximal function, introduced in \cite{KL}. 
Analogously, a discrete $\gamma$-median maximal function can be defined by taking a supremum of the discrete $\gamma$-median convolutions (see, for example, \cite{HKT}).

We fix a scale $r>0$ and cover the space $X$ with a countable family of balls $\{B_i\}=\{B(x_i,r)\}$, so that the enlarged balls are of bounded overlap. This means that there is a constant $C(c_d)>0$, depending only on the doubling constant, such that
\[
\sum_{i=1}^{\infty}\ch{2B_i}(x)\le C(c_d)<\infty
\]
for all $x \in X$. Then, a partition of unity related to the covering $\{B_i\}$ is constructed. There exist $C/r$-Lipschitz functions $\ph_i$, $i = 1,2,\ldots$, such that $0\le\ph_i\le 1$, 
$\ph_i=0$ outside $2B_i$ and $\ph_i\ge C^{-1}$ on $B_i$ for all $i$ and $\sum_{i=1}^{\infty}\ph_i=1$. 
Let $0<\gamma\le1/2$. 
A \emph{discrete $\gamma$-median convolution} of a function $u\in L^0(X)$ at scale $r>0$ is
\[
u_r^\gamma (x)=\sum_{i=1}^{\infty}m_{u}^\gamma(B_i)\ph_i (x),
\]
for all $x \in X$, where the balls $B_i$ and functions $\ph_i$ are as above.

We apply the next theorem, by which locally Lipschitz functions are dense in $N^s_{p,q}(X)$, to show that for compact sets we can restrict the set of admissible functions in the definition of the $C^s_{p,q}$-capacity to locally Lipschitz functions when $0<s<1$ and $0< p, q < \infty$.

\begin{thm}\label{dc1}{\cite[Theorem 1.1]{HKT}}\label{density of loc lip} 
Let $0<\gamma\le 1/2$, $0<s<1$, $0<p,q<\infty$ and $u\in \dot N^s_{p,q}(X)$. 
Then, the discrete $\gamma$-median convolution approximations 
$u_{2^{-i}}^\gamma$ converge to $u$ in $N^s_{p,q}(X)$ as $i\to\infty$. 
\end{thm}

\begin{thm}\label{dc2} Let $0<s<1$, $0<p,q<\infty$ and let $K\subset X$ be a compact set. Then
\[
C_{p,q}^s(K)\approx\inf\{\|u\|_{N^s_{p,q}(X)}^p: u\in\tilde\cA(K)\},
\]
where $\tilde\cA(K)=\{u\in\cA(K): u \text{ is locally Lipschitz}\}$.
\end{thm}

\begin{proof} Since $\tilde\cA(K)\subset \cA(K)$, it suffices to prove the ''$\ge$'' part.
Let $u\in \cA(K)$. Then there is an open set $U\supset K$ such that $u\ge 1$ in $U$. 
Let $V=\{x: d(x,K)<d(K,X\setminus U)/2\}$.
If $x\in V$ and $r<d(K,X\setminus U)/8$, then $B(y,2r)\subset U$ whenever $x\in B(y,2r)$.
It follows that $u^\gamma_r\ge 1$ in $V$ when $r<d(K,X\setminus U)/8$. Thus, $u^\gamma_r\in \tilde\cA(K) $, for small $r$, and so, by Theorem \ref{density of loc lip},
\[
\begin{split}
\inf\{\|v\|_{N^s_{p,q}(X)}^p: v\in\tilde\cA(K)\}&\le \liminf_{i\to\infty} \|u^\gamma_{2^{-i}}\|_{N^s_{p,q}(X)}^p\\
&\le \liminf_{i\to\infty} C(\|u\|_{N^s_{p,q}(X)}^p+\|u^\gamma_{2^{-i}}-u\|_{N^s_{p,q}(X)}^p)\\
&\le C\|u\|_{N^s_{p,q}(X)}^p.
\end{split}
\] 
The claim follows by taking infimum over $u\in \cA(K)$.
\end{proof}

\section{Netrusov--Hausdorff content}

In this section, we define a modified version of the Netrusov--Hausdorff content and prove lower bound and upper bound estimates for the Besov capacity in terms of this Netrusov--Hausdorff cocontent. 
The Netrusov--Hausdorff content was first used by Netrusov in \cite{Ne92} and \cite{Ne96} when studying the relations between capacities and Hausdorff contents in $\mathbb{R}^n$. 
We modify this content by taking the sum over the measures of the balls in the covering divided by the values $\phi(r_j)$ of the radii, where $\phi$ is an increasing function. In the setting of a doubling metric measure space, this kind of modification, instead of summing the powers of the radii of the balls in the covering, is natural since the dimension of the space is usually not (even locally) constant. 

\begin{defin}\label{HN content}
Let $\phi:(0,\infty)\to (0,\infty)$ be an increasing function and let $0<\theta<\infty$ and $0<R<\infty$. The {\em Netrusov--Hausdorff cocontent} of a set $E\subset X$ is
\[
\cH^{\phi,\theta}_R(E)
=\inf\Bigg[\sum_{i:2^{-i}<R}\bigg(\sum_{j\in I_i} \frac{\mu(B(x_j,r_j))}{\phi(r_j)} \bigg)^{\theta}\Bigg]^{1/\theta},
\]
where the infimum is taken over all coverings $\{B(x_j,r_j)\}$ of $E$ with $0<r_j\le R$ and $I_i=\{j: 2^{-i}\le r_j<2^{-i+1}\}$.
When $R=\infty$, the infimum is taken over all coverings of $E$ and the first sum is over $i\in \z$.
When $\phi(t)=t^d$, we use the notation $\cH^{d,\theta}_R:=\cH^{\phi,\theta}_R$.
\end{defin}

Notice that if the measure $\mu$ is (Ahlfors) $Q$-regular, that is, there is a constant $C>1$, such that $$C^{-1}r^Q \leq \mu(B(x,r)) \leq Cr^Q$$ for every $x \in X$ and $0<r<\diam(X)$, then the cocontent $\cH^{d,\theta}_R$ is comparable (with two-sided inequalites) with the $(Q-d)$-dimensional Netrusov--Hausdorff content defined using the powers of radii.

A similar modification of the classical Hausdorff content is standard in the metric setting. The {\em Hausdorff content of codimension $d$}, $0<d<\infty$, is
\[
\cH^{d}_R(E)=\inf\bigg\{\sum_{j=1}^{\infty} \frac{\mu(B(x_j,r_j))}{r_j^d}\bigg\},
\]
where $0<R<\infty$, and the infimum is taken over all coverings $\{B(x_j,r_j)\}$ of $E$ satisfying $r_j\le R$ for all $j$.
When $R=\infty$, the infimum is taken over all coverings $\{B(x_j,r_j)\}$ of $E$. Naturally,
the {\em Hausdoff measure of codimension $d$} is defined as
\[
\cH^d(E)=\lim_{R\to0}\cH^d_R(E).
\] 

We use the following Leibniz type rule for fractional $s$-gradients, and its corollary, in the proofs of Theorem \ref{Cap < NH} and Theorem \ref{toinen suunta}.

\begin{lem}{\cite[Lemma 3.10 and Remark 3.11]{HIT}}\label{LemWithLipForTL}
Let $0<s<1$, $0<p<\infty$ and $0<q\le\infty$, and let $S\subset X$ be a measurable set.
Let $u\colon X\to\re$ be a measurable function with $(g_k)_{k \in \mathbb{Z}}\in \D^s(u)$ and
let $\ph$ be a bounded $L$-Lipschitz function supported in $S$.
Then sequences $(h_k)_{k\in\z}$ and $(\rho_k)_{k\in\z}$, where
\begin{align*}
&\rho_k=\big(g_k\Vert \ph\Vert_{\infty}+2^{k(s-1)}L|u|\big)\ch{\operatorname{supp}\ph} \quad\text{and} \\
&h_k=\big(g_k+2^{s k+2}|u|\big)\Vert \ph\Vert_{\infty}\ch{\operatorname{supp}\ph}
\end{align*}
are fractional $s$-gradients of $u\ph$.
Moreover, if $u\in N^s_{p,q}(S)$, then $u\ph\in N^s_{p,q}(X)$ and $\|u\ph\|_{N^s_{p,q}(X)}\le C\|u\|_{N^s_{p,q}(S)}$.
\end{lem}

By choosing $u \equiv 1$ and $g_k \equiv 0$ for all $k \in \mathbb{Z}$ in (the proof of) the previous lemma, we obtain norm estimates for Lipschitz functions.

\begin{cor}\cite[Corollary 3.12]{HIT}\label{liplemma}
Let $0<s<1$, $0<p<\infty$ and $0<q\le\infty$.
Let  $\ph\colon X\to\re$ be an $L$-Lipschitz function supported in a bounded set $F\subset X$.
Then $\ph\in N^{s}_{p,q}(X)$ and
\begin{equation}\label{Msps norm u}
\|\ph\|_{N^{s}_{p,q}(X)}\le C(1+\|\ph\|_{\infty})(1+L^{s})\mu(F)^{1/p},
\end{equation}
where the constant $C>0$ depends only on $s$ and $q$.
\end{cor}

In the next theorem, we show that the Besov capacity of a set $E \subset X$ is bounded from above by a constant times the Netrusov--Hausdorff cocontent of the set $E$.

\begin{thm}\label{Cap < NH}
Let $0<s<1$, $0<p<\infty$, $0<q\leq \infty$, $E \subset X$ and $R\le 1$. Then there is a constant $C>0$ such that
$$C_{p,q}^s(E)\le C\cH^{sp,\theta}_R(E),$$
where $\theta=\min\{1,q/p\}$.
\end{thm}
\begin{proof}
Let $\{B(x_j,r_j)\}$ be a covering of the set $E$ such that $r_j\le 1$ for all $j$. Let $i \in \mathbb{Z}_{+} \cup\{0\}$ and
\[
u_i(x)=\max\{0,1-2^id(x,\cup_{j\in I_i}B(x_j,r_j))\},
\] 
where $I_i=\{j: 2^{-i}\le r_j<2^{-i+1}\}$. Then $u_i=1$ in $\cup_{j\in I_i}B(x_j,r_j)$, $u_i=0$ outside $\cup_{j\in I_i}B(x_j,2^{-i+2})$ and $u_i$ is Lipschitz with constant $2^i$. Since $i \geq 0$, we have that $1+2^{is} \leq C2^{is}$ and it follows from Corollary \ref{liplemma} and the doubling property that
\[
\begin{split}
C_{p,q}^s(\cup_{j\in I_i}B(x_j,r_j))&\le \|u_i\|_{N^s_{p,q}(X)}^p \\
&\le C \, (1+ ||u_i||_{\infty})^p (1+2^{is})^p \mu(\cup_{j\in I_i}B(x_j,2^{-i+2})) \\
&\le C \, 2^{isp}\mu(\cup_{j\in I_i}B(x_j,2^{-i+2})) \\
&\le C\sum_{j\in I_i}\frac{\mu(B(x_j,r_j))}{r_j^{sp}}.
\end{split}
\]
Let $\theta=\min\{1,q/p\}$. By Theorem \ref{r-subadd}, we have that
\begin{align*}
C_{p,q}^s(E)&\le C \Big(\sum_{i} C_{p,q}^s(\cup_{j\in I_i}B(x_j,r_j))^{\theta}\Big)^{1/\theta} \\
&\le C\Big(\sum_{i} \Big(\sum_{j\in I_i}\frac{\mu(B(x_j,r_j))}{r_j^{sp}} \Big)^{\theta}\Big)^{1/\theta}
\end{align*}
and the claim follows by taking the infimum over all covers $\{B(x_j,r_j)\}$ of the set $E$. 
\end{proof}

Next, we prove a converse estimate which gives a lower bound estimate for the capacity in terms of the Netrusov--Hausdorff cocontent. 

\begin{thm}\label{toinen suunta}
Let $0<s<1$, $0<p<\infty$, $0<q\leq \infty$ and let $\phi$: $(0,\infty) \rightarrow (0,\infty)$ be an increasing function, such that 
\[
\int_0^a\phi(t)^{-1/p}t^{s-1}\,dt<\infty
\]
for every $0<a<\infty$. Let $x_0 \in X$, $0<R<\infty$ and assume that $B(x_0,8R)\setminus B(x_0,4R)$ is nonempty. Then there are constants $C>0$ and $c>0$ such that
\[
\cH^{\phi, q/p}_{cR}(E)\le C C_{p,q}^s(E)
\]
for every compact set $E\subset B(x_0,R)$.
\end{thm}

\begin{rem}\label{neteuc}
For example, when $\phi(t)=t^d$, we have that $$\int_0^a t^{-d/p+s-1}\,dt<\infty,$$ if and only if $1-s+d/p < 1$. That is, $d<sp$. 
\end{rem}

\begin{proof} To avoid some inessential technical difficulties and make the notation more simple, we assume that $R=2^{-m}$, for $m \in \mathbb{Z}$. In the light of our proof, we can see that the result for $0<R<\infty$ can be obtained using the same argument. 

Let $\eps>0$ and $E \subset B(x_0,2^{-m})$ be a compact set. 
By Theorem \ref{dc2}, there is a locally Lipschitz function $v \in N^s_{p,q}(X)$, such that $v \geq 1$ on a neighbourhood of $E$ and 
$$\|v\|_{N_{p,q}^s(X)}^p < C C_{p,q}^s(E)+\eps.$$
Let $\psi$ be a Lipschitz function, such that $\psi =1$ on $B(x_0,2^{-m})$ and $\psi= 0$ outside $B(x_0,2^{-m+1})$. Then, $u=v\psi \in N_{p,q}^s(X)$ is Lipschitz continuous and $u\geq1$ on a neighbourhood of $E$. By Lemma \ref{LemWithLipForTL}, there exists $(g_k)_{k \in \mathbb{Z}}\in\D^s(u)$, such that $g_k=0$ outside $B(x_0,2^{-m+1})$, for every $k$, and
\begin{equation}\label{grad arvio2}
\|(g_k)\|_{l^q(L^p(X))}^p \leq C \|v\|_{N_{p,q}^s(X)}^p< C (C_{p,q}^s(E)+\eps).
\end{equation}
To be precise, we have here the fractional $s$-gradient of $u$, still denoted by $g_k$, which satisfies the Sobolev--Poincar\'e type inequality \eqref{joo} that is later used in the proof. By formula \eqref{good grad}, $g_k$ is supported in $B(x_0,2^{-m+1})$, for every $k$, and by \eqref{jep2} the inequality \eqref{grad arvio2} is satisfied.

Let $x\in E$ be a generalized Lebesgue point of $u$ (see Definition \ref{gLp}). Since $u$ is continuous, it follows from h) of Lemma \ref{median lemma} that every point in $E$ is such a point. Then,
\begin{equation}\label{jotain}
1\le u(x)\le |u(x)-m^\gamma_u(B(x,2^{-m}))|+|m^\gamma_u(B(x,2^{-m}))|.
\end{equation}
We can estimate the first term by Lemma \ref{median lemma}, properties d), f) and c) of $\gamma$-median, and by a telescoping argument 
\begin{align*}
|u(x)-m_{u}^{\gamma} (B(x,2^{-m}))| &\leq \sum_{k \geq m} |m_{u}^{\gamma} (B(x,2^{-k-1})) - m_{u}^{\gamma} (B(x,2^{-k}))| \\
&\leq \sum_{k \geq m} m_{|u-m_{u}^{\gamma} (B(x,2^{-k}))|}^{\gamma} (B(x,2^{-k-1})) \\
&\leq \sum_{k \geq m} m_{|u-m_{u}^{\gamma} (B(x,2^{-k}))|}^{\gamma/C} (B(x,2^{-k})).
\end{align*}
Then, it follows from Theorem \ref{mm} and Remark \ref{mm2} that
$$\sum_{k \geq m} m_{|u-m_{u}^{\gamma} (B(x,2^{-k}))|}^{\gamma/C} (B(x,2^{-k})) \leq C \sum_{k \geq m} 2^{-ks} \Big( \vint{B(x,2^{-k+1})} g_k^{p} \, d\mu \Big)^{1/p}.$$

Next, we estimate the second term of \eqref{jotain}.
Let $y\in B(x,2^{-m})\setminus F$, where $F$ is the exceptional set from the Definition \ref{fract grad}. Since $B(x,2^{-m})\subset B(x_0,2^{-m+1})$ and $B(x_0, 2^{-m+3})\setminus B(x_0, 2^{-m+2})$ is nonempty, there exists $z\in (B(x_0,2^{-m+3})\setminus B(x_0,2^{-m+2}))\setminus F$ such that $2^{-m}\le d(y,z)\ < 2^{-m+4}$. 
We define $g=\max\{g_k: m-4 \le k\le m-1\}$. Now,
\begin{align*}
|u(y)|&=|u(y)-u(z)|\le d(y,z)^s(g(y)+g(z)) \\
&=d(y,z)^s g(y) \le 2^{(-m+4)s} g(y)
\end{align*}
and by f), b), e) and g) of Lemma \ref{median lemma}, we have that
\begin{align*}
|m^\gamma_u(B(x,2^{-m}))| &\le m^{\gamma}_{2^{(-m+4)s} g}(B(x,2^{-m})) \\
&\le C \, 2^{-ms} \, m^\gamma_g(B(x,2^{-m})) \\
&\le C \sum_{k=m-4}^{m-1} 2^{-ks} \Big(\vint{B(x,2^{-m})}g_k^p\,d\mu\Big)^{1/p} \\
&\le C \sum_{k \geq m}2^{-ks}\Big(\vint{B(x,2^{-k+1})}g_k^p\,d\mu\Big)^{1/p}.
\end{align*}
Hence,
\[
\begin{split}
1&\le C \sum_{k \geq m}2^{-ks}\Big(\vint{B(x,2^{-k+1})}g_k^p\,d\mu\Big)^{1/p} \\
&\le C \, \Big(\sum_{k \geq m}\phi(2^{-k+1})^{-1/p} \, 2^{-ks}\Big)\sup_{k \geq m}\phi(2^{-k+1})^{1/p}\Big(\vint{B(x,2^{-k+1})}g_k^p\,d\mu\Big)^{1/p} \\
&\le C \, \Big(\int_0^{2^{-m+1}}\phi(t)^{-1/p} \, t^{s-1}\,dt \Big)\sup_{k \geq m}\phi(2^{-k+1})^{1/p}\Big(\vint{B(x,2^{-k+1})}g_k^p\,d\mu\Big)^{1/p}\\
&\le C\sup_{k \geq m}\phi(2^{-k+1})^{1/p}\Big(\vint{B(x,2^{-k+1})}g_k^p\,d\mu\Big)^{1/p}.
\end{split}
\]
Now, for every $x\in E$, there is a ball $B(x,2^{-k_x+1})$, such that
\[
\frac{\mu(B(x,2^{-k_x+1}))}{\phi(2^{-k_x+1})}\le C\int_{B(x,2^{-k_x+1})}g_{k_x}^p\,d\mu.
\]
By the $5r$-covering lemma, there exists a countable family of disjoint balls $B_j=B(x_j, 2^{-k_{x_j}+1})$, of radii $r_j=2^{-k_{x_j}+1} \leq 2^{-m}$, such that the dilated balls $5B_j$ cover the set $E$. We use notation $j \in I_i$, when $2^{-i}\leq 5r_j < 2^{-i+1}$. Then $k_{x_j}=i+3$, for $j \in I_i$, and since $\phi$ is increasing
$$\sum_{j \in I_i} \frac{\mu(5B_j)}{\phi(5r_j)} \leq C \sum_{j \in I_i} \frac{\mu(B_j)}{\phi(r_j)} \leq C \sum_{j \in I_i} \int_{B_j}g_{i+3}^p\,d\mu \leq C \, \|g_{i+3}\|_{L^p(X)}^p,$$
where we also used doubling and the disjointness of the balls $B_j$. Summing over $i$, we obtain 
$$\sum_{2^{-i}<5\cdot 2^{-m}} \Big( \sum_{j \in I_i} \frac{\mu(5B_j)}{\phi(5r_j)} \Big)^{q/p} \leq C \sum_{i \in \mathbb{Z}} \|g_{i+3}\|_{L^p(X)}^q$$
and it follows that
$$\cH^{\phi, q/p}_{5\cdot 2^{-m}}(E)\leq C \Big( \sum_{i \in \mathbb{Z}} \|g_{i+3}\|_{L^p(X)}^q \Big)^{p/q}.$$ Now, letting $\epsilon \rightarrow 0$ in \eqref{grad arvio2} proves the claim.
\end{proof}


\subsection*{Acknowledgements}
I am grateful to Toni Heikkinen for helpful comments and ideas during the preparation of this work. This research was supported by the Academy of Finland (grant no. 272886).

\end{document}